\documentclass[11pt,a4paper]{amsart}
\usepackage{a4wide}
\usepackage{subfiles}
\usepackage{lipsum}
\usepackage{hyperref}

\hypersetup{urlcolor=blue, citecolor=blue, linkcolor=blue, colorlinks=true}

\usepackage{enumitem}
\usepackage{amsmath}
\usepackage{amssymb}
\usepackage{mathtools}
\usepackage{amsfonts}
\usepackage{microtype}
\usepackage{multicol}
\usepackage{subcaption}
\usepackage{xcolor}
\usepackage{stackengine,scalerel}
\usepackage{tikz}
\usepackage{csquotes}
\usepackage{tikz-cd}
\usepackage{amsthm}
\usepackage{tcolorbox}
\usepackage[capitalise]{cleveref}
\usepackage[utf8]{inputenc}
\usepackage[english]{babel}

\DeclareMathAlphabet{\mathpzc}{OT1}{pzc}{m}{it}

\usepackage{newunicodechar}
\usepackage[style=alphabetic,backend=bibtex,maxnames=5,maxalphanames=5,language=english,doi=false,isbn=false,url=true]{biblatex}
\bibliography{literature.bib}
  
\renewbibmacro{in:}{}
\renewbibmacro*{issue+date}{%
  \ifboolexpr{not test {\iffieldundef{year}} or not test {\iffieldundef{issue}}}
    {\printtext[parens]{%
       \iffieldundef{issue}
         {\usebibmacro{date}}
         {\printfield{issue}%
          \setunit*{\addspace}%
          \usebibmacro{date}}}}
    {}%
  \newunit}

\let\oldtocsection=\tocsection
\let\oldtocsubsection=\tocsubsection
\let\oldtocsubsubsection=\tocsubsubsection
\renewcommand{\tocsection}[2]{\hspace{0em}\oldtocsection{#1}{#2}}
\renewcommand{\tocsubsection}[2]{\hspace{1em}\oldtocsubsection{#1}{#2}}
\renewcommand{\tocsubsubsection}[2]{\hspace{2em}\oldtocsubsubsection{#1}{#2}}

\newtheorem{bigthm}{Theorem}

\newtheorem{thm}{Theorem}[section]
\newtheorem{lem}[thm]{Lemma}
\newtheorem{prop}[thm]{Proposition}
\newtheorem{cor}[thm]{Corollary}

\theoremstyle{definition}

\theoremstyle{remark}

\newtheorem{rem}[thm]{Remark}

\newtheorem{construction}[thm]{Construction}

\setenumerate[1]{leftmargin=*,labelindent=0pt,label=(\roman*),}

\newcommand{\pref}[2]{\hyperref[#2]{#1 \ref*{#2}}}
\usepackage{todonotes}

\newcommand{\id}{\ensuremath{\operatorname{id}}}

\newcommand{\Spin}{\ensuremath{\operatorname{Spin}}}

\newcommand{\Diff}{\ensuremath{\operatorname{Diff}^{\scaleobj{0.8}{+}}}}

\newcommand{\Diffuo}{\ensuremath{\operatorname{Diff}}}

\newcommand{\BDiff}{\ensuremath{\operatorname{BDiff}^{\scaleobj{0.8}{+}}}}
\newcommand{\BDiffuo}{\ensuremath{\operatorname{BDiff}}}

\newcommand{\ind}{\mathrm{ind}}

\newcommand{\catsingle}[1]{\ensuremath{\mathcal{#1}}}

\newcommand{\cA}{\ensuremath{\catsingle{A}}}

\newcommand{\cL}{\ensuremath{\catsingle{L}}}

\newcommand{\Maps}{\ensuremath{\operatorname{Maps}}}

\newcommand{\ra}{\rightarrow}

\newcommand{\too}{\longrightarrow}
\newcommand{\congarrow}{\overset{\cong}\longrightarrow}
\newcommand{\embeds}{\hookrightarrow}

\newcommand{\xra}[1]{\xrightarrow{#1}}

\newcommand{\actson}{\curvearrowright}
\DeclarePairedDelimiter{\scpr}{\langle}{\rangle}

\newcommand{\calR}{\mathcal{R}}
\newcommand{\calA}{\mathcal{A}}

\newcommand{\calH}{\mathcal{H}}
\newcommand{\calF}{\mathcal{F}}
\newcommand{\bbN}{\mathbb{N}}
\newcommand{\bbH}{\mathbb{H}}
\newcommand{\bbP}{\mathbb{P}}
\newcommand{\bbZ}{\mathbb{Z}}
\newcommand{\bbQ}{\mathbb{Q}}
\newcommand{\bbR}{\mathbb{R}}

\newcommand{\inv}{{\mathrm{inv}}}

\newcommand{\nnsec}{{\mathrm{Sec}\ge0}}
\newcommand{\prc}{{\mathrm{Ric}>0}}
\newcommand{\nnrc}{{\mathrm{Ric}\ge0}}
\newcommand{\psc}{{\mathrm{scal}>0}}
\newcommand{\nnsc}{{\mathrm{scal}\ge0}}

\begin{document}
\author{Georg Frenck}
\email{\href{mailto:georg.frenck@kit.edu}{georg.frenck@kit.edu}}
\email{\href{mailto:math@frenck.net}{math@frenck.net}}
\address{Institut f\"ur Algebra und Geometrie, Englerstr.~2, 76131 Karlsruhe, Germany}

\author{Jens Reinhold} 
\email{\href{mailto:jens.reinhold@uni-muenster.de}{jens.reinhold@uni-muenster.de}}
\address{Mathematisches Institut, Einsteinstr. 62, 48149 M{\"u}nster, Germany}

\subjclass[2010]{53C21, 57R20, 57R22, 58D17.}

\thanks{GF: I am supported by the DFG (German Research Foundation) -- 281869850 (RTG 2229). 
\newline
JR: I am supported by the DFG (German Research Foundation) -- SFB 1442 427320536, Geometry:~Deformations and Rigidity, as well as under Germany's Excellence Strategy EXC 2044 390685587, Mathematics M\"unster:~Dynamics-Geometry-Structure}

\title[Bundles with non-multiplicative $\hat{A}$-genus and spaces of metrics]{Bundles with non-multiplicative $\hat{A}$-genus and spaces of metrics with lower curvature bounds}

\begin{abstract} 
We construct smooth bundles with base and fiber products of two spheres whose total spaces have non-vanishing $\hat{A}$-genus. We then use these bundles to locate non-trivial rational homotopy groups of spaces of Riemannian metrics with lower curvature bounds for all $\Spin$-manifolds of dimension six or at least ten which admit such a metric and are a connected sum of some manifold and $S^n \times S^n$ or $S^n \times S^{n+1}$, respectively. We also construct manifolds $M$ whose spaces of Riemannian metrics of positive scalar curvature have homotopy groups that contain elements of infinite order which lie in the image of the orbit map induced by the push-forward action of the diffeomorphism group of $M$. 
\end{abstract}

\maketitle

\section{Introduction}

\noindent 
For a closed smooth manifold $M$, let $\calR(M)$ denote the space of metrics on $M$, equipped with the $C^{\infty}$-topology. Its subspace $\calR_\psc(M)$ of metrics of positive scalar curvature has been the subject of much research in recent years. For example, it has been shown by Botvinnik--Ebert--Randal-Williams in \cite{BERW} that its homotopy groups are, roughly speaking, at least as complicated as the real $\mathrm{K}$-theory of a point, provided that $M$ is $\Spin$ and of dimension at least $6$. For further results, see \cite{walsh_hspaces,crowleyschicksteimle, erw_psc2}. 
In contrast, little is known about the topology of the corresponding spaces for positive (or nonnegative) Ricci (or sectional) curvature if they are non-empty; in particular, even simple non-vanishing results for its higher rational homotopy or homology groups are scarce. 

In this paper, we detect elements of infinite order in these groups for a large class of manifolds. To state our results in the greatest generality, let $\calF(M)\subset\calR(M)$ be a $\Diffuo(M)$-invariant subset that admits a $\Diffuo(M)$-equivariant continuous map $\iota_\calF\colon\calF(M)\to\calR_\psc(M)$. Furthermore, let $W_g^{2n}\coloneqq (S^n\times S^n)^{\#g}$ denote the $g$-fold connected sum of $S^n\times S^n$, and analogously $W_g^{2n+1}\coloneqq (S^n\times S^{n+1})^{\#g}$. For a closed manifold $M$ of dimension $d$, we define the \emph{genus of $M$} to be the largest number $g$ such that there exists a manifold $N$ with $M\cong N\#W_{g}^{d}$. The term $\pi_j(X) \otimes \bbQ \neq 0$ for some space $X$ shall mean that there exists a base point $x \in X$ such that  $\pi_j(X,x) \otimes \bbQ \neq 0$. 

\begin{bigthm}\label{thm:main1}
	Let $d\ge10$ with $d\not=13$ and let $M$ be a $d$-dimensional $\Spin$-manifold of genus at least $1$. If $\calF(M)\not=\emptyset$, then either $\pi_1(\calF(M))$ is infinite or $\pi_j(\calF(M))\otimes\bbQ\not=0$ for some $2\le j\le 9$.
\end{bigthm}

\begin{rem}
\begin{enumerate}
	\item The most obvious examples for $\calF(M)$ are $\Diffuo(M)$-invariant subsets of $\calR_\psc(M)$, since the inclusion is obviously continuous and $\Diffuo(M)$-equivariant. This includes the space $\calR_\prc(M)$ of metrics of positive Ricci curvature. By applying the Ricci-flow one can show that the space $\calR_\nnsc(M)$ of metrics of nonnegative scalar curvature also admits such a map, provided that $M$ does not admit a Ricci-flat metric (cf. \pref{Proposition}{prop:ricciflow}). In this case \pref{Theorem}{thm:main1} also holds for $\Diffuo(M)$-invariant subsets of the space $\calR_\nnsc(M)$, which yields results on the spaces $\calR_\nnrc(M)$ and $\calR_\nnsec(M)$ of nonnegative Ricci and sectional curvature.
	\item One can refine the statement in this theorem so that one obtains a more precise estimate on which rational higher homotopy groups are nontrivial depending on the dimension modulo 8:
	\begin{align*}
		&d\equiv 0\ (\mathrm{mod}\;8):\quad j\in \{2,3,7\} &&	d\equiv 1\ (\mathrm{mod}\;8):\quad j\in \{2,3,6\}\\
		&d\equiv 2\ (\mathrm{mod}\;8):\quad j\in \{2,5\} &&	d\equiv 3\ (\mathrm{mod}\;8):\quad j\in \{2,4\}\\
		&d\equiv 4\ (\mathrm{mod}\;8):\quad j=3 &&	d\equiv 5\ (\mathrm{mod}\;8):\quad j\in \{2,3,4,6\}\\
		&d\equiv 6\ (\mathrm{mod}\;8):\quad j\in\{2,3,4,9\} &&	d\equiv 7\ (\mathrm{mod}\;8):\quad j\in \{2,3,4,8\}
	\end{align*} 
\end{enumerate}
\end{rem}

\begin{rem}
	Botvinnik--Ebert--Wraith derive a similar result in \cite{bew}, but in contrast to ours it neither applies to manifolds of dimension $d \equiv 6 \ (\text{mod }8)$, nor to odd-dimensional ones. We also improved the lower bound required on the dimension and genus of the manifold.
\end{rem}

\noindent 
\pref{Theorem}{thm:main1} does not apply to dimensions smaller than 10 and its assertion is weakest in dimensions $d\equiv 6\ (\mathrm{mod}\;8)$, but in those dimensions we get the following additional result that also holds in dimension 6. 

\begin{bigthm}\label{thm:main2}
	Let $d\equiv 6\ (\mathrm{mod}\;8)$ be a positive integer, and let $M$ be a $d$-dimensional $\Spin$-manifold of genus at least $1$. If $\calF(M)\not=\emptyset$, then at least one of the following is true.
\begin{enumerate}
		\item The map $\iota_\calF\colon\calF(W_1^d) \to \calR_\psc(W_1^d)$ collapses infinitely many path components to one.  
		\item $H_1(\calF(M);\mathbb Q) \neq 0$.\qedhere
	\end{enumerate}
\end{bigthm}

\begin{rem}
	\begin{enumerate}
		\item Note that statement $(ii)$ from \pref{Theorem}{thm:main2} is stronger than $\pi_1(\calF(M))$ being infinite. It implies that there is an element of infinite order in the abelianisation of this group.
		\item For the case of positive Ricci curvature, a list of manifolds to which \pref{Theorem}{thm:main1} and \pref{Theorem}{thm:main2} are applicable is given in \cite[Corollary 1.2]{bew} or \cite[Definition 1.1]{burdick}.
	\end{enumerate}
\end{rem}

\begin{rem}\label{rem:dim13}
	For $d=13$ (and in fact for every $d\ge13$ with $d\equiv 5\ (\mathrm{mod}\;8)$) we get a result that looks like a mixture of \pref{Theorem}{thm:main1} and \pref{Theorem}{thm:main2}. In this case at least one of the following is true.
	\begin{enumerate}
		\item The map $\iota_\calF\colon\calF(W_1^d) \to \calR_{\psc}(W_1^d)$ collapses infinitely many path components to one. 
		\item $\pi_1(\calF(M))$ is infinite.
		\item $\pi_2(\calF(M))\otimes\bbQ\neq0$.\qedhere
	\end{enumerate}
\end{rem}

\noindent 
The two theorems above follow from a more general statement about the rational cohomology groups of the space $\calF(M)$ (cf. \pref{Theorem}{thm:mainprc} and \pref{Proposition}{prop:implyingpsc}). In its proof, the main new ingredient is the construction of $W_1^d$-bundles over the product of two spheres with non-vanishing  $\hat A$-genus. The most general version of our construction yields the following.

\begin{prop}\label{prop:bundleintro}
For positive integers $p,q,i,j$ such that $2j < p < 4i$ and $2i < q < 4j$, there exists a smooth bundle $E \to S^{4i-p} \times S^{4j-q}$ with fiber $S^p \times S^q$, containing a trivialized disk bundle, whose total space admits a $\Spin$-structure and has non-vanishing $\hat{A}$-genus. 
\end{prop}

\noindent 
Employing these bundles in another way, we also obtain examples of manifolds $M$ for which the image of the map on homotopy groups induced by the orbit map $\Diffuo(M,D)\to\calR_\psc(M)$ (associated to the push-forward action) contains elements of infinite order. Here $\Diffuo(M,D)$ is the topological group of diffeomorphisms of $M$ fixing an embedded disk $D\subset M$ of codimension zero.

\begin{bigthm}\label{thm:main3}
	Let $i,j$ be positive integers with $|i-j|<\min(i,j)$ and $k<4j-2i-1$. Then for every $p\in\{2j+1,\dots, 4i-1\}$ there exists a fiber bundle $S^p\times S^{4j-k-1}\to M\to S^{4i-p}$ with $\calR_\psc(M)\not=\emptyset$ such that the image of the map induced by the orbit map
	$\pi_k(\Diffuo(M,D),\id)\too\pi_k (\calR_\psc(M),g)$
	contains an element of infinite order for every $g\in\calR_\psc(M)$.\end{bigthm}

\begin{rem}
	Note that this gives examples for every $k\ge0$ and since $4j-k-1>2i\ge\min(p,4i-p)$ the manifold $M$ is $\min(p,4i-p)$-connected. It follows from the proof of  \pref{Theorem}{thm:main3} pertaining to $\Diffuo(M,D)$ that it is also true for $N\# M$ for \emph{any} $\Spin$-manifold $N$ of positive scalar curvature, by extending diffeomorphisms by the identity.
\end{rem}	
	
\begin{rem}
Manifolds as in \pref{Theorem}{thm:main3} were known to exist for some time by \cite[Corollary 2.6]{HankeSchickSteimle}, however their construction \enquote{is based on abstract existence results in differential topology [and] does not yield an explicit description of the diffeomorphism type of the [...] manifold} {[}loc.~cit.~p.~3{]}. 
	
Recently, Kupers, Krannich, and Randal-Williams have shown \cite{KKRW} that the image of the map $\pi_3(\Diffuo(\bbH\bbP^2), \text{id})\to\pi_3(\calR_\psc(\bbH\bbP^2),g_{\text{st}})$ contains an element of infinite order, analogous to the manifolds from above. Here, $g_{\text{st}}$ denotes the standard metric on $\bbH\bbP^2$, which is of positive sectional curvature. Their result is especially remarkable as it goes beyond the abstract result by Hanke, Schick and Steimle insofar it also gives elements in $\pi_3(\calR_\prc(\bbH\bbP^2),g_{\text{st}})$ and even $\pi_3(\calR_{\mathrm{sec}>0}(\bbH\bbP^2), g_{\text{st}})$ of infinite order. We do not know if the manifolds from \pref{Theorem}{thm:main3} admit a metric of positive Ricci curvature.
\end{rem}

\section{Bundles over products of spheres with non-vanishing $\hat{A}$-genus}
\label{sec:bundles}
\subsection{The construction of the bundle}

In this section, we construct smooth bundles whose fiber and base are products of spheres and whose total spaces have non-vanishing $\hat{A}$-genus, thereby proving \pref{Proposition}{prop:bundleintro}. Throughout the whole section, we assume that $p,q,i,j$ are positive integers such that the inequalities $2j < p < 4i$ and $2i < q < 4j$ are satisfied.

\begin{construction}\label{constr:bundle}
We build a smooth bundle 
	\[(S^p\times S^q)\setminus{\mathring{D}^{p+q}}\to E' \overset{\pi}\too S^{4i-p}\times S^{4j-q}\]
	as follows. Consider the plumbing of two trivial disk bundles, 
	\[S^{p} \times D_-^{q}  			\cup_{D_-^{p} \times D_-^{q}} D_-^{p} \times S^{q} \cong S^{p}\times S^{q} \backslash\mathring{D}^{p+q},\]
	where $D^d_{-}$ denotes the lower hemisphere of $S^d$. 
	See \pref{Figure}{fig:bundle-nonclosed} for a visualization of this manifold.
The mapping spaces $\Omega^p_{\text{sm}} SO(q)$ and $			\Omega_{\text{sm}}^q SO(p)$ of those maps \[(S^p,D^p_-) \to (SO(q),1) \quad 		\text{and} \quad (S^q,D^q_-) \to (SO(p),1)\] so that the adjoint maps $S^p \times 	D^{q} \to D^q$ and $D^p \times S^q \to D^q$ are smooth can be endowed with the 	compact-open topology and the group structure induced from point-wise composition. The adjoint maps then twist the handles of $S^{p}\times S^{q} \backslash\mathring{D}^{p+q}$, which produces two commuting actions of $\Omega^p_{\text{sm}} SO(q)$ and $\Omega_{\text{sm}}^q SO(p)$ on $S^p \times S^q \backslash\mathring{D}^{p+q}$.
	The inclusions $\Omega_{\text{sm}}^q SO(p) \subset \Omega^q SO(p)$ and $	\Omega_{\text{sm}}^p SO(q) \subset \Omega^p SO(q)$ are weak equivalences.
	We continue by picking $\alpha \in \pi_{4i}(BSO(q))$ and $\beta \in \pi_{4j} (BSO(p))$ such that $p_i(\alpha), p_j(\beta) \neq 0$, where $p_i$ denotes the $i$-th Pontryagin class \cite{MilnorStasheff} that gives rise to a homomorphisms $\pi_i BSO(d) \to \mathbb Q$. This is possible by the assumptions on $p,q,i,j$, cf.~\cite[Lemma 5]{milnor_structures}. We then choose representatives of clutching function of the corresponding vector bundles, 
$\hat\alpha\colon S^{4i-p-1}\to \Maps((S^p,D^p_-), (SO(q),1))$ and $\hat \beta\colon S^{4j-q-1}\to \Maps((S^q,D^q_-), (SO(p),1))$, respectively. By the discussion above we may assume that these maps are smooth. Finally, we define for $x\in S^{4i-p-1}$ and $y\in S^{4j-q-1}$.
	\begin{align*}
		&\alpha_x\colon S^p\times D_-^q \too S^p\times D_-^q && \alpha_x(s,t) = \hat\alpha(x)(s)\cdot t\\
		&\beta_y\colon D_-^p \times S^q \too D_-^p \times S^q && \beta_y(s,t) = \hat\beta(y)(t)\cdot s.
	\end{align*}
	Both families fix the disk $D^p_- \times D^q_-$ and hence they can be extended to $S^{p} \times D_-^{q}  \cup_{D_-^{p} \times D_-^{q}} D_-^{p} \times S^{q} \cong S^{p}\times S^{q} \backslash\mathring{D}^{p+q}$. Since they have disjoint support, they obviously commute, and $(\alpha_x)_{x\in S^{4i-p-1}}$, $(\beta_y)_{y\in S^{4j-q-1}}$ can be seen as clutching functions of $S^{p}\times S^{q} \backslash\mathring{D}^{p+q} $ with disjoint supports, so they produce a smooth bundle 
	\[(S^p\times S^q)\setminus{\mathring{D}^{p+q}}\to E' \overset{\pi}\too S^{4i-p}\times S^{4j-q}\]
which contains a trivialized $D^{p+q}$-subbundle. Also, without loss of generality we may assume that the boundary subbundle of $E'$ contains a trivial $D^{p+q-1}$-subbundle. This can be achieved by changing $\alpha$ and $\beta$ by appropriate homotopies.
\end{construction}

\begin{figure}[ht]
	\begin{center}
	\includegraphics[width=0.4\textwidth]{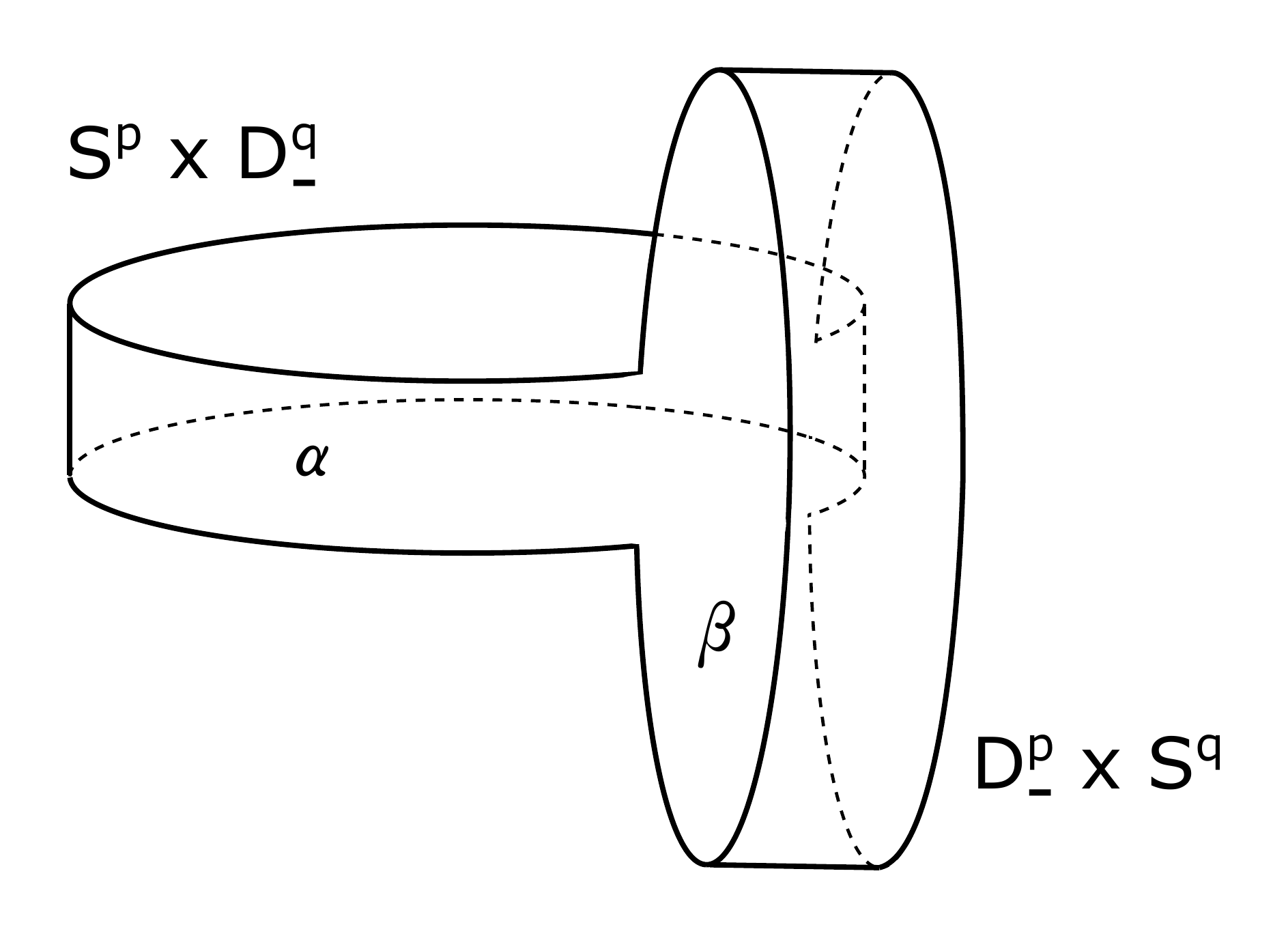}
	\caption{The manifold $S^p \times S^q \backslash\mathring{D}^{p+q}$ serves as the fiber}\label{fig:bundle-nonclosed}
	\end{center}
	\end{figure}

\begin{lem}\label{lem:closing} Let $\pi\colon X \to B$ be a smooth $(S^n,D^n)$-bundle whose base space $B = S^{i_1} \times \dots \times S^{i_r}$ is a product of spheres. Then the following holds.
\begin{enumerate}
\item If $n$ is even and $\dim(B) \leq n-5$, there exists a map $\varphi\colon B \to B$ of non-zero degree such that $\varphi^{\ast} \pi$ is trivial.
\item If $n$ is odd and $\dim(B) \leq n-3$, there exists a map $\varphi\colon B \to B$ of non-zero degree such that $\varphi^{\ast} \pi$ is the bundle of boundaries of a smooth $D^{n+1}$-bundle.
\end{enumerate}
\end{lem}

\noindent 
\pref{Lemma}{lem:closing}, which will be proven below, implies that we can modify the construction of the bundle above so that we obtain a bundle $S^p\times S^q\to E\to S^{4i-p}\times S^{4j-q}$ with closed fiber $S^p \times S^q$ by precomposing the classifying map of $E'$ with a self-map $\varphi$ of $S^{4i-p}\times S^{4j-q}$ of nonzero degree, and gluing in a disk bundle along the boundary of $\varphi^*E'$. This boundary is a smooth $(S^{p+q-1},D^{d+q-1}_-)$-bundle over $S^{4i-p} \times S^{4j-q}$. 

It follows from the inequalities on $p,q,i,j$ that $2i \leq q-1$ and $2j \leq p-1$. If $d = p+q-1$ is even (and thus one of $p-1, q-1$ is odd), then $2(i+j) \leq p+1-3$ and hence $\dim(S^{4i-p} \times S^{4j-q}) = 4(i+j) - (p+q) \leq p+q-6 = (d-1)-5$ holds, as (i) in \pref{Lemma}{lem:closing} demands with $n = d-1$. If $d = p+q-1$ is odd, then $\dim(S^{4i-p} \times S^{4j-q}) = 4(i+j) - (p+q) \leq p+q-4 = (d-1)-3$ holds as (ii) demands.

\noindent 
For the proof of \pref{Lemma}{lem:closing} we need the following auxiliary result.

\begin{lem} \label{lem:HomThTorus}
Let $B = S^{i_1} \times \dots \times S^{i_r}$ be a product of spheres, and let $Y$ be a connected based space such that every element in $\pi_j(Y,y)$ has finite order for $1 \leq j \leq \dim(B) = i_1 + \dots + i_r$. Then for any map $f\colon B \to Y$, there exists a map $\varphi \colon B \to B$ of non-zero degree such that $f \circ \varphi$ is nullhomotopic.
\end{lem}
\begin{proof}
We first prove the claim for the case that $Y = K(\pi,j)$ is an Eilenberg--MacLane space, with $\pi$ a finite group (abelian if $j \geq 2$) and $1 \leq j \leq b$. If $j = 1$, homotopy classes of based maps $f\colon B \to Y$ are in bijection with group homomorphisms $ \bbZ^{\ell} \to \pi$, where $\ell$ is the number of indices $i_k$ that are $1$. Since every element in $\pi$ has finite order, the image of this group homomorphism is finite and precomposing with a suitable map $\varphi \colon B \to B$ of non-zero degree induces the constant homomorphism $\bbZ^{\ell} \to \pi$, which implies that $f\circ \varphi$ is null. If $2 \leq j \leq \dim(B)$, homotopy classes of maps $B \to Y$ are in bijection with $H^j(B;\pi)$, and we can find a suitable map $\varphi \colon B \to B$ of non-zero degree that induces multiplication by a multiple of the order of $\pi$ on $j$th cohomology, which implies that $f\circ \varphi$ is null.

Employing an inductive argument over the Whitehead tower of $Y$ that decomposes into fiber sequences $Y\langle j +1\rangle \to Y\langle j \rangle \to K(j,\pi_j(Y))$, we can precompose with suitable maps $\varphi_j \colon B \to B$ of non-zero degree such that $f \circ \varphi_1 \circ \dots \circ \varphi_j \colon B \to Y$ can be lifted along $Y\langle j+1\rangle \to Y\langle j \rangle$. At the end, we have produced a map $\varphi = \varphi_1 \circ \dots \circ \varphi_{\dim(B)} \colon B \to B$ of non-zero degree such that $f \circ \varphi$ can be lifted along $Y\langle \dim(B) +1\rangle \to Y$. But any map $B \to Y\langle \dim(B)+1 \rangle$ is null, which finishes the proof.
\end{proof}

\begin{proof}[Proof of \pref{Lemma}{lem:closing}] The given bundle $\pi$ is classified by a map $f\colon B \to B\!\Diffuo(S^n,D_-^n) = B\!\Diffuo_{\partial}(D^{n}_+)$, well-defined up to homotopy. If $n$ is even, then it follows from \cite[Theorem 4.1]{RW_UpperRange} that the assumption on $Y$ in \pref{Lemma}{lem:HomThTorus} above is satisfied for $Y = \BDiffuo_{\partial}(D^{n})$, so (i) follows from this lemma.

For (ii), we first note that the fiber sequence 
\begin{equation*}
B\!\Diffuo_{\partial}(D^{n+1}) \to B\!\Diffuo(D^{n+1},D^n_-) \to B\!\Diffuo_{\partial}(D^{n}_+),
\end{equation*}
where $D^n_+$ denotes the upper hemisphere of $\partial D^{n+1} = S^n$, can be delooped with respect to the canonical $E_d$ structure \cite{BDiff-delooping}, which gives rise to an exact sequence 
\begin{equation*} \label{SES_Mapping_Spaces}
[B,B\!\Diffuo(D^{n+1},D^n_-)] \to [B,B\!\Diffuo_{\partial}(D^{n})] \xrightarrow{\delta} [B,B^2\Diffuo_{\partial}(D^{n+1})] 
\end{equation*}
of pointed sets of homotopy classes.

Since $n$ is odd and hence $n+1 \geq 6$ is even, it follows from \cite[Theorem 4.1]{RW_UpperRange} as above that the homotopy groups of the space $B^2\!\Diffuo_{\partial}(D^{n+1})$ only contain elements of finite order in degrees at most $n-3$. Thus, \pref{Lemma}{lem:HomThTorus} above implies that there exists a map $\varphi\colon B \to B$ of non-zero degree such that $\delta(f\circ \varphi) = \delta (f) \circ \varphi= 0$, hence $\varphi^{\ast} \pi$ is the bundle of boundaries of a smooth $D^{n+1}$-bundle, as claimed.
\end{proof}

\subsection{Characteristic Classes}
In this section, we prove that the total space of the fiber bundle $S^p \times S^q \to E \to S^{4i-p} \times S^{4j-q}$ constructed above has non-vanishing $\hat{A}$-genus.

\begin{lem} \label{lem:ConstructionBundle1} 
	The only two potentially non-vanishing Pontryagin-numbers of $E$ are $p_{i+j}$ and $p_ip_j$, and the latter is actually non-zero. Furthermore, $\hat{A}(E) \neq 0$.\end{lem}
\begin{proof} 
To verify the first claim, it suffices to consider the $S^{p} \times S^{q} \backslash\mathring{D}^{p+q}$-bundle $E_0 \to S^{4i-p} \times S^{4j-q} \backslash\mathring{D}^{4(i+j)-(p+q)} =: B  \simeq S^{4i-p} \vee S^{4j-q}$ obtained from $E \to S^{4i-p} \times S^{4j-q}$ previously constructed by cutting out a disk in the base, and show that the only non-zero Pontryagin classes of $E_0$ are $p_i$ and $p_j$, as all non-zero Pontryagin classes of $E$ besides $p_{i+j}$ remain non-zero under the map $H^{\ast}(E,\mathbb Q) \to H^{\ast}(E_0,\mathbb Q)$. For this note that the core $S^p \vee S^q \subset S^{p} \times S^{q} \backslash\mathring{D}^{p+q}$ is a deformation retract and is fixed under the handle twisting in \pref{Construction}{constr:bundle}, hence the bundle $E_0 \to S^{4i-p} \vee S^{4j-q}$ is trivial as a fibration. 
 
The submanifolds $(\star \vee S^{q}) \times (S^{4i-p} \vee \star)$ and $(S^{p}\vee \star) \times (\star\vee S^{4j-q})$ have a trivial normal bundle in $E$, hence the stable tangent bundle of $E$ is trivial when pulled back along the inclusion of these submanifolds. Moreover, the compositions
\[(S^{p} \vee \star) \times (S^{4i-p} \vee \star)\ra S^{4i}\xra{\alpha} BSO \quad \text{and} \quad (\star \vee S^{q}) \times (\star \vee S^{4j-q})\ra S^{4j}\xra{\beta} BSO\] classify the stabilized normal bundles of these spheres, where the first maps are given by collapsing the $(4i-1)$- resp. $(4j-1)$-skeleton. This implies that the only non-zero Pontryagin classes of $E_0$ are $p_i$ and $p_j$, and that these classes evaluate nontrivially when paired with the fundamental classes of the submanifolds $(\star \vee S^{q}) \times (S^{4i-p} \vee \star)$ and $(S^{p}\vee \star) \times (\star\vee S^{4j-q})$. As these submanifolds intersect transversally in one point in $E$, we deduce that $p_ip_j(E)$ is non-zero.

The signature of any manifold that fibers over a sphere vanishes by \cite{ChernHirzebruchSerre} for $d\ge2$ and by \cite[Theorem I.2.2]{meyer_faserbuendel} for $d=1$. The claim $\hat{A}(E) \neq 0$ then follows from the proceeding \pref{Lemma}{lem:Ahat_nonzero}.
\end{proof}

\begin{lem}\label{lem:Ahat_nonzero}
 Let $M$ be a closed oriented manifold of dimension $4(i+j)$ such that all characteristic numbers of $M$ except $p_{i+j}$ and $p_ip_j$ vanish, and $p_ip_j(M) \neq 0$. Then if the signature $\sigma(M)$ vanishes, $\hat{A}(M) \neq 0$.
\end{lem}
\begin{proof}
We will show that the corresponding polynomials $\cL_{i+j}$ and $\hat{\cA}_{i+j}$ are linearly independent in $\mathbb Q[p_i,p_j,p_{i+j}]^{4(i+j)}$, which implies the assertion. Let us denote the coefficients of the $\cL$- and $\hat{\cA}$-polynomials as given in the following expressions.
\begin{align*}
\cL &=1 + \frac{1}{3}p_1 + \dots + s_ip_i + \dots + s_jp_j + \dots + s_{i+j}p_{i+j} + \dots + s_{i,j}p_ip_j + \dots
\end{align*}
\begin{equation*}
\hat{\cA} = 1 - \frac{1}{24}p_1 + \dots + a_ip_i + \dots + a_jp_j + \dots + a_{i+j}p_{i+j} + \dots + a_{i,j}p_ip_j + \dots
\end{equation*}
We claim that\footnote{For $\cL$, this identity is mentioned in the appendix of a preliminary version of \cite{Dalian}.}  $s_is_j = s_{i+j} + \lambda s_{i,j}$ and $a_ia_j = a_{i+j} + \lambda a_{i,j}$ where $\lambda =2$ if $i=j$ and $\lambda =1$ if $i\neq j$. This can be derived from the the multiplicativity of the corresponding sequences \cite{MilnorStasheff}, as follows. We consider the multiplicative sequence arising from $\cL$, abusively denoted by the same variable, where we abbreviate the sequence $(0, \dots, 0,1,0,\dots)$ with the $1$ at the $\ell$th position by $e_{\ell}$.

Suppose first that $i = j$. Then
\begin{align*}
	s_i^2 + 2s_{i,i} &= \left((1+ s_ie_i + s_{i,i}e_{2i} + \dots)^2\right)_{2i} \\
	&= \left( \cL(1+ e_i)\cL(1+ e_i)\right)_{2i} \\
		&=\cL((1+ e_i)(1+ e_i))_{2i} \\
		&= \cL(1+ 2\cdot e_i + e_{2i})_{2i} \\
		&= s_{2i} + 4s_{i,i} ,
\end{align*}
and subtracting $2s_{i,i}$ gives the desired equation.

If $i \neq j$, the calculation is slightly different. Without loss of generality, we may assume $i < j$. Furthermore, we will first assume that $i$ does not divide $j$, in which case we have
\begin{align*}
	s_is_j &= \left((1+ s_ie_i + s_{i,i}e_{2i} + \dots)(1+ s_je_j + s_{j,j}e_{2j} + \dots)\right)_{i+j} \\
	&= \left( \cL(1+ e_i)\cL(1+ e_j)\right)_{i+j} \\
		&=\cL((1+ e_i)(1+ e_j))_{i+j} \\
		&= \cL(1+ e_i + e_j + e_{i+j})_{i+j} \\
		&= s_{i+j} + s_{i,j},
\end{align*}
as claimed. If $i$ does divide $j$, however, and $r \coloneqq j/i$ is integral, this computation has to be changed in such a way that $s_{i,i, \dots, i}$ (with $i$ appearing $r$ times) is added to both terms at the very beginning and end of the equation, which does not affect the truth of the identity. The proof of the identity for the coefficients of $\hat{\cA}$ is completely analogous. 

It remains to show that the matrix \begin{equation*}
\begin{pmatrix}
s_{i,j} & s_{i+j} \\
a_{i,j} & a_{i+j} \\
\end{pmatrix}
\end{equation*}
or equivalently the matrix
\begin{equation*}
\begin{pmatrix}
s_is_j & s_{i+j} \\
a_ia_j & a_{i+j} \\
\end{pmatrix}
\end{equation*} 
is non-singular. Since for all $k \geq 0$,
\[s_k = \frac{2^{2k}(2^{2k-1}-1)B_{2k}}{(2k)!} \quad \text{and} \quad a_k = - \frac{B_{2k}}{2 \cdot (2k)!}\] 
(see \cite[Ch.\,1,3]{Hirzebruch}), it thus suffices to check that the matrix 
\begin{equation*}
\begin{pmatrix}
(2^{2i-1}-1)(2^{2j-1}-1) & 2^{2(i+j)-1}-1 \\
\frac{1}{4} & -\frac{1}{2} \\
\end{pmatrix}
\end{equation*} 
is non-singular, which is clear since its determinant $-\frac{1}{4}(2^{2i}-1)(2^{2j}-1)$ is always negative. 
\end{proof}

\begin{rem}
For a far more general formula of the coefficients of $\cL$- and $\hat{\cA}$-polynomials, see \cite{Fowler--Su, BergBerg}.
\end{rem}
\noindent 
In order to make use of the bundle $E \to S^{4i-p} \times S^{4j-q}$ from above with applications to positive curvature in mind, we need the existence of a $\Spin$-Structure on its vertical tangent bundle, or equivalently the total space of $E$. This follows from the next lemma since the base space $S^{4i-p} \times S^{4j-q}$ is $\Spin$ and the fiber $S^p \times S^q$ is $2$-connected as $p, q > 2$ by assumption.
\begin{lem}\label{lem:exspinstr}
For a $2$-connected smooth $d$-manifold $F$ and a $\Spin$-manifold $B$, the total space of any oriented smooth bundle $E \to B$ with fiber $F$ that contains a trivialized disk bundle admits a $\Spin$-structure.
\end{lem}
\begin{proof}
Let $\iota\colon B\times D^d\embeds E$ be the bundle embedding of trivial disk bundle. Then $B=B\times\{0\}\to B\times D^d \to E$ is a section of $E$. The map $E\to B$ is $3$-connected and hence $\iota$ induces an isomorphism $H^2(B\times D^d)\to H^2(E)$. Therefore we have
\[\iota^*w_2(E) = w_2(\iota^*E) = w_2(T(B\times D^d)) = w_2(B) = 0\]
and injectivity of $\iota$ implies that $w_2(E)$ vanishes.
\end{proof}

\begin{rem}
	In \pref{Lemma}{lem:exspinstr}, the assumption that $E$ contains a trivial disk bundle is essential: Let $\pi\colon V\to S^2$ be a vector bundle of rank $n\ge3$ with nontrivial second Stiefel--Whitney class and let $S^{n-1}\to E\overset{\pi}\too S^2$ be the embedded sphere bundle. The stabilized tangent bundle of $E$ is given by 
	\[TE\oplus\underline\bbR^2\cong \underbrace{(\pi^*TS^2\oplus\underline\bbR)}_{\cong\underline\bbR^3}\oplus \underbrace{(T_\pi E\oplus\underline\bbR)}_{\pi^*V}\]
	and hence $w_2(E) \not=0$ and $E$ does not admit a $\Spin$-structure.
\end{rem}

\noindent
The following proposition combines what we proved so far and is the essential result needed for the geometric applications in \pref{Section}{sec:applications}.
\begin{prop}\label{prop:spin-structures}
	The total space $E$ from \pref{Lemma}{lem:ConstructionBundle1} admits no metric of positive scalar curvature.
\end{prop}
\begin{proof}
By the above discussion (\pref{Lemma}{lem:exspinstr}), there is a $\Spin$-structure on $E$. Using that we know $\hat{A}(E)\neq 0$ by \pref{Lemma}{lem:ConstructionBundle1}, the statement about metrics of positive scalar curvature is now a well-known consequence of the Lichnerowicz formula \cite{lichnerowicz} and the Atiyah--Singer index theorem \cite{atiyahsinger_index}.
\end{proof}

\noindent 
Recall the definition of the generalized Miller-Morita-Mumford-classes $\kappa_c \in H^{\ast}(B\Diff(M))$, also known as tautological or simply $\kappa$-classes: for a smooth $M$-bundle $\pi\colon E \to B$, let $T_{\pi}E$ denote the vertical tangent bundle. Then for any class $c \in H^{\ast}(BSO(\dim(M)))$, $\kappa_{c} := \pi_! c(T_{\pi}E) \in H^{\ast-\dim(M)}(B)$, where $\pi_!$ stands for the Gysin map\footnote{This is often called fiber integration.}. Note that for the bundle $\pi\colon E\to S^{4i-p}\times S^{4j-q}$, we have
\begin{align*}
	\scpr{\kappa_{\hat\calA_k}(\pi),[S^{4i-p}\times S^{4j-q}]} &= \scpr{\pi_! \hat\calA_k(T_{\pi}E),[S^{4i-p}\times S^{4j-q}]}\\
		&=\scpr{\hat\calA(T(S^{4i-p}\times S^{4j-q}))\cup\pi_! \hat\calA_k(T_{\pi}E),[S^{4i-p}\times S^{4j-q}]}\\
		&=\scpr{\hat\calA(\pi^*T(S^{4i-p}\times S^{4j-q}))\cup \hat\calA_k(T_\pi E), [E]}\\	
		&=\scpr{\hat\calA(TE),[E]} = \hat A(E)\not=0.
\end{align*}
where $\hat\calA_k$ is the degree $k$ homogeneous part of the $\hat\calA$-class. We conclude that $0 \neq \kappa_{\hat\calA_k}(E) \in H^{4k-d}(S^{4i-p}\times S^{4j-q};\bbQ)$. Since $E$ contains a trivial $D^d$-subbundle we deduce the following by glueing  trivial bundles onto $\pi$ 

\begin{cor}\label{cor:bundle-exists}
For any $\Spin$-manifold $M$ of dimension $d = p+q$, there exists a smooth bundle $\pi_M \colon E_M \to S^{4i-p}\times S^{4j-q}$ with fiber $(S^p\times S^q)\#M$ whose vertical tangent bundle has $\Spin$-structure and which satisfies $\kappa_{\hat\calA_k}(\pi_M) \neq 0$.
\end{cor}

\section{Applications to spaces of metrics}\label{sec:applications}
\subsection{Metrics on $S^p\times S^q$}\label{sec:prc}
\noindent
For the implications on the space $\calF(M\#(S^p\times S^q))$, we use a slight upgrade of the detection principle from \cite{bew} where it is only stated for positive Ricci curvature. To make it easier to follow the line of thought, it seems most natural to include the proof from loc.~cit.~and adjust it slightly.

Let $M$ be a $d$-dimensional closed $\Spin$-manifold and let $\calR_\inv(M)$ denote the space of Riemannian metrics on $M$ such that the associated Dirac operator is invertible. Let us assume that $\calR_\inv(M)$ is nonempty and let $\pi\colon E\to B$ be an $M$-bundle over a connected space $B$ with a $\Spin$-structure on the vertical tangent bundle. We get an associated $\Diffuo(M)$-principal bundle $Q\to B$ such that $Q\times_{\Diffuo(M)}M\cong E$ and we define
\[\calR_{\inv}(\pi)\coloneqq Q\underset{\Diffuo(M)}{\times}\calR_\inv(M)\overset{\Pi}\too B\]
\begin{thm}\cite[Theorem 2.1]{bew}\label{thm:bew}
	In the situation described above, let us assume that
	\begin{enumerate}
		\item The action of $\pi_1(B)$ on $H^0(\calR_\inv(M))$ induced by the fiber transport factors through a finite group.
		\item For some $k>\frac d4$ the class $\kappa_{\hat\calA_k}(E)\in H^{4k-d}(B;\bbQ)$ is nontrivial.
	\end{enumerate}
	Then there exists an $r\in\{2,\dots,4k-d\}$ such that $H^{4k-d-r}(B;\bbQ)\not=0$ and $H^{r-1}(\calR_\inv(M);\bbQ)\not=0$. 
\end{thm}
\begin{proof}
	If the action $\pi_1(B)\actson H^0(\calR_\inv(M))$ factors through a finite group there is a finite cover $p\colon \tilde B\to B$ such that the action $\pi_1(\tilde B)\actson H^0(\calR_\inv(M))$ is trivial. The induced map $p^*\colon H^{k}(B;\bbQ)\to H^k(\tilde B;\bbQ)$ is injective and so the bundle $p^*E\to \tilde B$ satisfies assumption $(ii)$, too. Therefore we may assume that the action $\pi_1(B)\actson H^0(\calR_\inv(M))$ is trivial.
	
	Consider the following diagram of fibrations. 
	\begin{center}
	\begin{tikzpicture}
		\node (1) at (0,0) {$\calR_\inv(M)$};
		\node (2) at (4,0) {$\calR_\inv(\pi)$};
		\node (3) at (8,0) {$B$};
		
		\node (4) at (4,1.2) {$\Pi^*E$};
		\node (5) at (8,1.2) {$E$};
		\node (6) at (4,2.4) {$M$};
		\node (7) at (8,2.4) {$M$};
		
		\draw[->] (1) to (2);
		\draw[->] (2) to node[above]{$\Pi$} (3);
		\draw[->] (4) to (5);
		\draw[-, line width=10pt, draw=white] (5) to (3);
		\draw[-, line width=10pt, draw=white] (4) to (2);
		\draw[->] (5) to node[right]{$\pi$} (3);
		\draw[->] (4) to (2);
		\draw[->] (6) to (4);
		\draw[->] (7) to (5);
	\end{tikzpicture}
	\end{center}
	The bundle $\Pi^*E$ admits a canonical fiberwise metric with invertible Dirac operators which implies that $\Pi^*\ind(E)\in KO^{-d}(\calR_\inv(\pi))$ vanishes. By the cohomological version of the Atiyah--Singer family index theorem we have
	\[\Pi^*\kappa_{\hat \calA_k} = (\mathrm{ch_{2k}}\circ c)(\Pi^*\ind(E))=0\in H^{4k-d}(\calR_\inv(\pi);\bbQ)\]
	where $\mathrm{ch_{2k}}$ denotes the $2k$-th component of the Chern character and $c$  is the complexification map. So the nontrivial class $\kappa_{\hat \calA_k}$ lies in the kernel of 
	\[\Pi^*\colon H^{4k-d}(B;\bbQ)\too H^{4k-d}(\calR_\inv(\pi);\bbQ).\]
	Consider the Serre spectral sequence for the fibration $\calR^{\inv}(\pi)\overset{\Pi}\too B$. The homomorphism $\Pi^*$ agrees with the composition
	\begin{align*}
		H^{4k-d}(B;\bbQ)&\too H^{4k-d}(B;\calH^0(\calR_\inv(M);\bbQ))=E_2^{4k-d,0} \\
			&\too E_{4k-d-1}^{4k-d,0}  \embeds H^{4k-d}(\calR_\inv(\pi);\bbQ).
	\end{align*}
	Since the action of $\pi_1(B)$ on $H^0(\calR_\inv(M);\bbQ)$ is trivial the coefficient system is simple and the first map is injective. So there is an element in the kernel of $E_2^{4k-d,0} \to E_{4k-d-1}^{4k-d,0}$ and hence there exists an $r\in\{2,\dots, 4d-k\}$ such that the  differential $d^r\colon E_r^{4k-d-r,r-1} \to E_{r}^{4k-d,0}$ is nontrivial. The map $H^{4k-d-r}(B;H^{r-1}(\calR_\inv(M);\bbQ)= E_2^{4k-d-r,r-1}\to E_r^{4k-d-r,r-1}\not=0$ is surjective and it follows that $0\not= H^{4k-d-r}(B;H^{r-1}(\calR_\inv(M);\bbQ)$ which enforces $H^{4k-d-r}(B;\bbQ)\not=0$ and $H^{r-1}(\calR_\inv(M);\bbQ)\not=0$.
\end{proof}

\begin{prop}\label{prop:implyingpsc}
	\pref{Theorem}{thm:bew} holds moreover for every nonempty, $\Diffuo(M)$-invariant subset $\calF(M)\subset \calR(M)$ which admits a $\Diffuo(M)$-equivariant, continuous map $\iota\colon\calF(M)\to \calR_\inv(M)$.
\end{prop}

\begin{proof}
	We show that any $M$-bundle $S\to X$ that admits a fiberwise metric $(g_x)_{x\in X}$ with $g_x\in \calF(\pi^{-1}(x))$ also admits a fiberwise metric with invertible Dirac operators: If $\alpha_x\colon M\cong\pi^{-1}(x)$ are some identifications of the fibers with $M$, then $\bigl((\alpha_x)_*\iota(\alpha_x^*g_x)\bigr)_{x\in X}$ is a fiberwise metric with invertible Dirac operators. Since $\iota$ is $\Diffuo(M)$-equivariant the above family does not depend on the choice of $\alpha_x$ and is hence well-defined. 
	
	Therefore the bundle $\Pi_\calF^*E\to\calF(\pi)$ analogous to the bundle $\Pi^*E\to\calR_\inv(\pi)$ from above admits a fiberwise metric with invertible Dirac operators and the rest of the proof goes through without change.
\end{proof}

\begin{prop}\label{prop:ricciflow}
	Every $\Diffuo(M)$-invariant subset of $\calF(M)\subset\calR_\psc(M)$ admits a map $\iota$ as in \pref{Proposition}{prop:implyingpsc}. If $M$ does not admit a Ricci-flat metric, the same holds for every nonempty, $\Diffuo(M)$-invariant subset of $\calF(M)\subset\calR_\nnsc(M)$.
\end{prop}

\begin{proof}
	By the Schr\"odinger--Lichnerowicz formula, the Dirac operator associated to a metric of positive scalar curvature on a closed $\Spin$-manifold is invertible. Since the inclusion $\calF(M)\subset\calR_\psc(M)$ is obviously $\Diffuo(M)$-equivariant, the first claim follows from \pref{Proposition}{prop:implyingpsc}.
	
	For the second claim, it suffices to construct a continuous, $\Diffuo(M)$-equivariant map 
	\[\iota\colon\calR_\nnsc(M)\to\calR_{\psc}(M).\]
	This is done by an argument similar to the one in \cite{schickwraith}. By \cite{bgi_ricciflow} the solution to the Ricci-flow depends continuously on the initial metric with respect to the $C^\infty$-topology on $\calR(M)$ and hence there exists a continuous map $f\colon\calR(M)\to(0,\infty)$ such that for every initial condition $g\in\calR(M)$ the solution to the Ricci flow exists for all $t<f(g)$. The Ricci-flow preserves isometries and therefore $f$ can be chosen to be $\Diffuo(M)$-invariant. Furthermore, if $g$ has nonnegative scalar curvature then for all $t>0$ the metrics $g_t$ have positive scalar curvature for all $0<t<f(g)$ by \cite[Proposition 2.18]{brendle} since $M$ does not admit a Ricci-flat metric. We define the map 
	\[\iota\colon\calR_\nnsc(M)\to\calR_\psc(M),\quad  g\mapsto g_{f(g)/2}.\]
	If $\alpha\colon M\congarrow M$ is a diffeomorphism, then $(\alpha^*g)_{f(\alpha^*g)/2} = (\alpha^*g)_{f(g)/2}=\alpha^*(g_{f(g)/2})$, again because the Ricci flow preserves isometries. Hence $\iota$ is $\Diffuo(M)$-equivariant and continuous.
\end{proof}

\noindent 
Using the bundles constructed in \pref{Section}{sec:bundles} we derive our main result: 

\begin{thm}\label{thm:mainprc}
	Let $p,q,i,j\in\bbN$ be such that $2j<p<4i$ and $2i<q<4j$. Moreover, let $M$ be a $\Spin$-manifold of dimension $(p+q)$ and let $\emptyset\not=\calF(M\#(S^p\times S^q))\subset\calR(M)$ be as in \pref{Proposition}{prop:implyingpsc}. If $4i-p$ and $4j-q$ are both at least $2$, then $H^{m}(\calF(M\#S^p\times S^q);\bbQ)\not=0$ for some $m\in\{4i-p-1,4j-q-1,4(i+j)-(p+q)-1\}$. If $4i-p=1$ and for every $\alpha\in\Diffuo(S^p\times S^q,D^{p+q})$ there exists an $l\in\bbN$ such that $(\alpha^l\#\id)^*\colon \calF(M\#(S^p\times S^q))\to\calF(M\#(S^p\times S^q))$ is homotopic to the identity, then $H^{m}(\calF(M\#S^p\times S^q);\bbQ)\not=0$ holds for some $m\in\{4j-q-1,4j-q\}$. If $4j-q=1$ as well, then $m=1$. 
\end{thm}

\begin{proof}
	By the assumptions on $p,q,i,j$ there exists a $M\#(S^p\times S^q)$-bundle $E\to S^{4i-p}\times S^{4j-q}$ with a $\Spin$-structure on the vertical tangent bundle and non-vanishing $\kappa_{\hat\calA_k}$ by \pref{Corollary}{cor:bundle-exists}. The claim then follows immediately from \pref{Theorem}{thm:bew} and \pref{Proposition}{prop:implyingpsc}. In the case $4i-p=1$, note that by assumption, the action $\pi_1(B)\actson H^0(\calF(M\#(S^p\times S^q)))$ factors through a finite group, since $\pi_1(B)$ is either given by $\bbZ$ or $\bbZ^2$.
\end{proof}

\begin{rem}
	With the precise statement it is possible to illustrate two more improvements over the result from \cite{bew}: Firstly we can give more precise estimates which rational cohomology groups are possibly nontrivial and secondly it is possible to show that multiple cohomology groups are nontrivial. The space $\calF(W_{1}^{42})$ for example has at least $5$ non-vanishing rational cohomology groups, namely one out of each of $\{2,5\},\;\{6,13\},\;\{10,21\},\;\{14,29\}$ and $\{18,37\}$.
\end{rem}

\noindent 
In order to deduce \pref{Theorem}{thm:main1} and \pref{Theorem}{thm:main2}, we will recall the rational Hurewicz theorem.

\begin{thm} (\cite{Serre}, see also \cite[Theorem 18.3]{MilnorStasheff})
	Let $X$ be a simply connected space with $\pi_i(X)\otimes \bbQ=0$ for $2\le i\le r$. Then the Hurewicz map induces an isomorphism
	\[H\colon\pi_i(X)\otimes\bbQ\too H_i(X;\bbQ)\]
	for $1\le i <2r+1$ and a surjection for $i=2r+1$.
\end{thm}

\begin{proof}[{Proof of \pref{Theorem}{thm:main1}}]
	We start by considering the case $d=7+8k$ for some $k\ge1$. Let us first assume that $\calF(M)$ is simply connected. For $p=4k+3$, $q=4k+4$, $i=k+2=j$, \pref{Theorem}{thm:mainprc} implies that $H^k(\calF(M);\bbQ)\not=0$ for some $k\in\{3,4,8\}$ and by the universal coefficient theorem $H_k(\calF(M);\bbQ)\not=0$. If $\pi_i(\calF(M))\otimes\bbQ=0$ for $i=2,3,4$, the rational Hurewicz theorem enforces $\pi_8(\calF(M))\otimes\bbQ\not=0$. If $\pi_1(\calF(M))$ is finite, consider the universal covering $p\colon\widetilde{\calF(M)}\to\calF(M)$. This yields a transfer map $p_!\colon H^k(\widetilde{\calF(M)};\bbQ)\to H^k(\calF(M);\bbQ)$ with $p_!\circ p^*=\id$ and hence $p_!$ is surjective.	If $\pi_i(\calF(M))\otimes\bbQ=0$ for $i=2,3,4$, the same holds for $\widetilde{\calF(M)}$. Therefore by the same argument as above we get $\pi_8(\calF(M))\otimes\bbQ=\pi_8(\widetilde{\calF(M)})\otimes\bbQ\not=0$. The other cases are completely analogous and the required choices for $p,q,i,j$ are given in the table below for $k\ge1$ (with the exception for $d=13$).
	
	\begin{tabular}{c || c | c | c | c | c | c | l}
		$d$ 	& $p$ 	& $q$ 	& $i$ 	& $j$ 	& $4i-p-1$ 	& $4j-q-1$		& $k\in$\\
		\hline
		$8k+2$	& $4k+1$	& $4k+1$	& $k+1$	& $k+1$	& $2$	& $2$		& $\{2,5\}$\\
		$8k+3$	& $4k+1$	& $4k+2$	& $k+1$	& $k+1$	& $2$	& $1$		& $\{2,4\}$\\
		$8k+4$	& $4k+2$	& $4k+2$	& $k+1$	& $k+1$	& $1$	& $1$		& $\{3\}$\\
		$8k+5$	& $4k+2$	& $4k+3$	& $k+1$	& $k+2$	& $1$	& $4$		& $\{4,6\}$\\
		$8k+6$	& $4k+3$	& $4k+3$	& $k+2$	& $k+2$	& $4$	& $4$		& $\{4,9\}$\\
		$8k+7$	& $4k+3$	& $4k+4$	& $k+2$	& $k+2$	& $4$	& $3$		& $\{3,4,8\}$\\
		$8k+8$	& $4k+4$	& $4k+4$	& $k+2$	& $k+2$	& $3$	& $3$		& $\{3,7\}$\\
		$8k+9$	& $4k+4$	& $4k+5$	& $k+2$	& $k+2$	& $3$	& $2$		& $\{2,3,6\}$\\
	\end{tabular}\\
	
	\noindent 
	Note that the case $d=13$ has to be excluded since it forces $j=2$ and hence $4j-q-1=0$. Therefore the assumption that the action of $\pi_0(\Diffuo(S^6\times S^7,D^{13}))$ on $\pi_0(\calF(S^6\times S^7))$ factors through a finite group is required in this case (cf. \pref{Remark}{rem:dim13}).
\end{proof}

\begin{proof}[Proof of \pref{Theorem}{thm:main2}]
	If the pullback action $\bbZ^2\cong\pi_1(B)\to\pi_0(\Diffuo(W_1^d,D))\actson \pi_0(\calF(W_1^d))$ factors through a finite group, then \pref{Theorem}{thm:mainprc} implies $H^1(\calF(M);\bbQ)\not=0$. If the pullback action does not factor through a finite group, then the orbit of this action gives an infinite family of components of $\calF(W_1^d)$. By \cite[Theorem B]{Frenck} the action is trivial on $\pi_0(\calR_\psc(W_1^d))$ or factors through $\bbZ/2$ depending on the dimension. Hence this infinite family of components of $\calF(W_1^d)$ gets mapped to the the same component by $\iota_\calF$.
\end{proof}

\subsection{The action of $\Diffuo(M)$ on $\calR_\psc(M)$}\label{sec:psc}
Let $F$ be a compact orientable manifold and let us consider an oriented fiber bundle of the form $F\to E\to S^{k}$ for $k\ge0$. This is classified by the homotopy class of a map $\alpha\colon S^{k-1}\to \Diff(F)$ based at the identity by clutching together two trivial $F$-bundles over the disk $D^{k}$. This means $E$ is given by 
\[E = \bigl(D^{k}\times F \amalg D^{k}\times F\bigr)/ (s,f) \sim (s,\alpha(s)(f))\]
for $s\in S^{k-1}$. Assuming that $F$ carries a metric $g_F$ of positive scalar curvature, we get an orbit map $\rho\colon\Diff(F)\to\calR_\psc(F)$ to the space of psc metrics on $F$ via pullback $\rho(f)=f^*g_F$. We have the following well-known observation (see e.g. \cite[Remark 1.5]{HankeSchickSteimle}). The proof consists of constructing a metric on $E$ that has fiberwise positive scalar curvature, shrinking the fibers and using the O'Neill formulas or by an explicit computation.

\begin{prop}\label{prop:criterion}
	If $\rho_*[\alpha] = 0\in\pi_{k-1}(\calR_\psc(F),g_F)$, then $E$ admits a psc-metric.
\end{prop} 

\noindent 
Now let $S^p\times S^q\to E\overset\pi\to S^{4i-p}\times S^{4j-q}$ be the bundle from \pref{Section}{sec:bundles}. Let $M\coloneqq \pi^{-1}(S^{4i-p}\times\{1\})$. Then $E$ is an $M$-bundle over $S^{4j-q}$ with a $\Spin$-structure on the vertical tangent bundle, non-vanishing $\hat A$-genus and hence no metric of positive scalar curvature. In order to apply \pref{Proposition}{prop:criterion} we need the existence of a positive scalar curvature metric on $M$ which is guaranteed by the following Lemma. 

\begin{lem}\label{lem:psc-existence}
	Let $F'\to F^d\to S^n$ be an oriented smooth fiber bundle with simply connected, stably parallelizable fiber $F'$ that admits a metric of positive scalar curvature. If $F$ is non-spinnable or if $F$ is $\Spin$ and $d\not\equiv1,2\ (\text{mod } 8)$, then $F$ admits a metric of positive scalar curvature as well. If $F$ is $\Spin$ and $d\equiv1,2\ (\text{mod }  8)$, then a sufficient condition for $F$ to admit a positive scalar curvature metric is that the element $\alpha \in \pi_n(\BDiff(F'))$ classifying the bundle is divisible by $2$.
\end{lem}

\begin{proof}
	For $n=1$ the Lemma follows in both cases from \cite{Frenck} as the mapping class group acts trivially on $\pi_0(\calR_\psc(F'))$ or factors through $\bbZ/2$ depending on the dimension. Hence the mapping torus admits a psc-metric under the named assumptions.  For $n\ge2$ we note that $F$ is simply connected and thus admits a psc-metric if it is either non-spinnable, or $\Spin$ and the $\alpha$-invariant of $F$ vanishes (cf. \cite{gromovlawson_classification} and \cite{stolz_simplyconnected}). This finishes the proof in the case that $F$ is non-spinnable, so let $F$ be $\Spin$ from now on. If $d\not\equiv1,2\ (\text{mod } 8)$, the $\alpha$-invariant agrees with the $\hat A$-genus which vanishes because stably parallelizable manifolds are $\hat A$-multiplicative by \cite[Proposition 1.9]{HankeSchickSteimle}. For the case $d\equiv1,2\ (\text{mod } 8)$ let $B\!\Diffuo^{\Spin}(F')$ denote the classifying space for $F'$-bundles with a $\Spin$-structure on the vertical tangent bundle (see \cite[Section 3.3]{ebert_characteristic} or \cite[Section 3.3]{Frenck} for a more detailed discussion.). By \cite[Lemma 3.3.6]{ebert_characteristic} the map $B\!\Diffuo^{\Spin}(F')\to B\!\Diff(F')$ induces an isomorphism on $\pi_n$ for $n\ge2$ because $F'$ is simply connected and hence has a unique $\Spin$-structure for a given orientation. Then $\alpha(F) \in KO^{-d}(*) \cong \bbZ/2\bbZ$ vanishes because of the given condition since the canonical map $\pi_n (\BDiffuo^{\Spin}(F')) \to \Omega^{\Spin}_n (\BDiffuo^{\Spin}(F')) \to \Omega^{\Spin}_{d}$ is a homomorphism. 
\end{proof}

\noindent
\pref{Theorem}{thm:main3} now follows immediately from \pref{Proposition}{prop:spin-structures}, \pref{Proposition}{prop:criterion} and \pref{Lemma}{lem:psc-existence}.

\begin{rem}
	\pref{Theorem}{thm:main3} holds for every Riemannian condition that is satisfied by $M$ and satisfies the hypothesis of \pref{Proposition}{prop:implyingpsc}. Therefore, having a more explicit construction compared to the one from \cite{HankeSchickSteimle} could yield non-triviality results for the induced map $\pi_k(\Diffuo(M))\otimes\bbQ\to\pi_k(\calR_C(M))\otimes\bbQ$.
\end{rem}

\subsection*{Acknowledgements} We thank Johannes Ebert and Manuel Krannich for fruitful discussions. We are also grateful to Philipp Reiser and especially Manuel Krannich for comments on a draft of this paper and to Anand Dessai for helpful remarks pertaining to the Ricci flow. Finally, we thank the anonymous referee for advice that strengthened our results and improved the presentation in the introduction.

\bigskip
\printbibliography
\bigskip
\end{document}